\documentclass[11pt, a4paper]{amsart}
\usepackage{amssymb} 
\usepackage{mathabx} 
\usepackage{epic}  
\usepackage[vcentermath]{youngtab} 
\usepackage[backref=page]{hyperref} 
\usepackage{amsthm}
\usepackage[capitalise]{cleveref} 

\makeatletter

\def\myMRbibitem{\@ifnextchar[\my@lbibitem\my@bibitem}

\def\mybiblabel#1#2{\@biblabel{{\hyperref{http://www.ams.org/mathscinet-getitem?mr=#1}{}{}{#2}}}}

\def\myhyperanchor#1{\Hy@raisedlink{\hyper@anchorstart{cite.#1}\hyper@anchorend}}

\def\my@lbibitem[#1]#2#3#4\par{%
    \item[\mybiblabel{#2}{#1}\myhyperanchor{#3}\hfill]#4%
    \@ifundefined{ifbackrefparscan}{}{\BR@backref{#3}}%
    \if@filesw{\let\protect\noexpand\immediate
       \write\@auxout{\string\bibcite{#3}{#1}}}\fi\ignorespaces%
}

\def\my@bibitem#1#2#3\par{%
    \refstepcounter\@listctr
    \item[\mybiblabel{#1}{\the\value\@listctr}\myhyperanchor{#2}\hfill]#3%
    \@ifundefined{ifbackrefparscan}{}{\BR@backref{#2}}%
    \if@filesw\immediate\write\@auxout
        {\string\bibcite{#2}{\the\value\@listctr}}\fi\ignorespaces%
}

\makeatother

\newtheoremstyle{note}
  {5pt}
  {5pt}
  {\small}
  {}
  {\bfseries}
  {.}
  {.5em}
  {}

\theoremstyle{plain}
    \newtheorem{thm}{Theorem}
    \newtheorem{prop}{Proposition}[section]

    \newtheorem{lemma}[prop]{Lemma}

\theoremstyle{definition}

\theoremstyle{remark}

\theoremstyle{note}
    
	\newtheorem{example}[prop]{Example}

\crefname{thm}{Theorem}{Theorems} 
\Crefname{thm}{Theorem}{Theorems} 
\crefname{otherthm}{Theorem}{Theorems}  
\Crefname{otherthm}{Theorem}{Theorems} 
\crefname{prop}{Proposition}{Propositions} 
\Crefname{prop}{Proposition}{Propositions}

\numberwithin{equation}{section}         

\hypersetup{bookmarksdepth=subsection} 

\renewcommand*{\backref}[1]{}
\renewcommand*{\backrefalt}[4]{\quad \tiny 
    \ifcase #1 (Not cited.)%
    \or        (Cited on page~#2.)%
    \else      (Cited on pages~#2.)%
    \fi}

\newcommand{\arxiv}[1]{Preprint \href{http://arxiv.org/abs/#1}{arXiv:{#1}}}

\crefname{section}{\S}{\S\S} 
\Crefname{section}{\S}{\S\S} 
\crefname{subsection}{\S}{\S\S} 
\Crefname{subsection}{\S}{\S\S}

\newcommand{\comment}[1]{}

\newcommand{\C}{\mathbb{C}}

\newcommand{\Z}{\mathbb{Z}}

\renewcommand{\setminus}{\smallsetminus}

\newcommand{\Mat}{\mathrm{Mat}}

\DeclareMathOperator{\codim}{codim}

\DeclareMathOperator{\rank}{rank}
\DeclareMathOperator{\col}{col}
\newcommand{\cupro}{\smallsmile}
\newcommand{\brickpattern}{\drawline(.25,0)(.25,.25)\drawline(.75,.25)(.75,.5)\drawline(.25,.5)(.25,.75)\drawline(.75,.75)(.75,1)\put(0,0){\grid(1,1)(1,.25)}}
\newcommand{\diagonalpattern}{\drawline(.75,0)(1,.25)\drawline(.5,0)(1,.5)\drawline(.25,0)(1,.75)\drawline(0,0)(1,1)\drawline(0,.25)(.75,1)\drawline(0,.5)(.5,1)\drawline(0,.75)(.25,1)}
\newcommand{\lightdiagonalpattern}{\drawline(.5,0)(1,.5)\drawline(0,0)(1,1)\drawline(0,.5)(.5,1)}

\title{Note on the dimension of certain algebraic sets of matrices}
\author[Bochi]{Jairo Bochi}
\address{Pontif\'icia Universidade Cat\'olica do Rio de Janeiro (PUC--Rio)}
\urladdr{www.mat.puc-rio.br/$\sim$jairo}
\email{jairo@mat.puc-rio.br}
\author[Gourmelon]{Nicolas Gourmelon}
\address{Institut de Math\'{e}matiques de Bordeaux, Universit\'{e} Bordeaux 1}
\email{nikolaz.gourmelon@gmail.com}
\date{\today}

\begin{document}

\maketitle

\section{Preamble}

In this short note we prove a lemma about the dimension of certain algebraic sets of matrices.
This result is needed in our paper \cite{BG_control}.
The result presented here has also applications in other situations
and so it should appear as part of a larger work \cite{BG_transitive}.

\section{Statement of the result}\label{s.statement}

%

If $A \in \Mat_{n \times m}(\C)$, let $\col A \subset \C^n$ denote the column space of $A$.
A set $X \subset \Mat_{n \times m}(\C)$ is called \emph{column-invariant}
if 
$$
\left.
\begin{array}{c}
A \in X \\ 
B \in \Mat_{n \times m}(\C)\\
\col A = \col B
\end{array}
\right\} \ \Rightarrow \ 
B \in X.
$$
So a column-invariant set $X$ is characterized by its set of column spaces.
We enlarge the latter set by including also subspaces, thus defining:
\begin{equation}\label{e.bracket_notation}
\ldbrack X \rdbrack := \big\{ E \text{ subspace of } \C^n ; \; E \subset \col A \text{ for some } A \in X \big\}.
\end{equation}
Then we have:

\begin{thm}\label{t.main}
Let $X \subset \Mat_{n \times m}(\C)$ be a nonempty algebraically closed, column-invariant set.
Suppose $E$ is a vector subspace of $\C^n$
that does not belong to $\ldbrack X \rdbrack$.
Then
$$
\codim X \ge m + 1 - \dim E \, .
$$
\end{thm}

It is obvious that the algebraicity hypothesis is indispensable. 


\Cref{t.main} follows without difficulty from intersection theory of the grassmannians
(``Schubert calculus'').
We tried to make the exposition the least technical as possible, to make it accessible to non-experts (like ourselves).

\section{A particular case}\label{s.particular}

Define
\begin{equation}\label{e.R_k}
R_k := \big \{ A \in \Mat_{n \times m}(\C) ; \; \rank A \le k \big\} \, .
\end{equation}
We recall (see \cite[Prop.~12.2]{Harris}) that 
this is an irreducible algebraically closed set of codimension
\begin{equation}\label{e.cod_R_k}
\codim R_k = (m-k)(n-k) \qquad \text{if } 0 \le k \le \min(m,n).
\end{equation}

\begin{proof}[Proof of \cref{t.main} in the case $E = \C^n$]
If $E = \C^n$ then the hypothesis $\C^n \not\in \ldbrack X \rdbrack$
means that $X \subset R_{n-1}$.
We can assume that $n-1 \le m$, otherwise the conclusion of the 
\lcnamecref{t.main} is vacuous.
Thus $\codim X \ge \codim R_{n-1} = m + 1 - n$, as we wanted to show.
\end{proof}

It does not seem likely that the general \cref{t.main}
can be reduced to \eqref{e.cod_R_k}.

\section{Reduction to a property of grassmannians} \label{s.reduction}

We will show that to prove \cref{t.main} 
it is sufficient to prove a dimension estimate (\cref{t.schubert} below) 
for certain subvarieties of a grassmaniann.

\subsection{Grassmannians}

Given integers $n > k \ge 1$, the \emph{grassmanniann} $G_k(\C^n)$ 
is the set of the vector subspaces of $\C^{n}$ of 
dimension $k$.

The grassmannian can be interpreted a subvariety of a higher dimensional complex projective space
as follows.
The \emph{Pl\"ucker embedding} is the map $G_k(\C^n) \to P(\bigwedge^k C^n)$
defined as follows: for each $V \in G_k(\C^n)$, take a basis $\{v_1, \dots, v_k\}$
and map $V$ to $[v_1 \wedge \cdots v_k]$.
This is clearly an one-to-one map.
It can be shown (see e.g.\cite[p.~61ff]{Harris}) that the image is 
an algebraically closed subset 
of $P(\bigwedge^k C^n)$.
Its dimension is 
\begin{equation}\label{e.dim_G}
\dim G_k(\C^n) = k(n-k).
\end{equation}

If $E \subset \C^n$ is a vector space with $\dim E = e \le k$ then 
we consider the following subset of $G_k(\C^n)$: 
\begin{equation}\label{e.special schubert}
S_k(E) := \big\{V \in G_k(\C^n) ; \; V \supset E \big\}.
\end{equation}
(This is a Schubert variety of a special type, as we will see later.)
Since any $V \in S_k(E)$ can be written as $E \oplus W$ for some $V \subset W^\perp$,
we see that $S_k(E)$ is homeomorphic to $G_{k-e}(\C^{n-e})$.

We will 
show that an algebraic set that avoids $S_k(E)$ cannot be too large:

\begin{thm}\label{t.schubert}
Fix integers $1 \le e \le k < n$. 
Suppose that $Y$ is an algebraically closed subset of $G_k(\C^n)$
that is disjoint from $S_k(E)$,
for some $e$-dimensional subspace $E \subset \C^n$.
Then $\codim Y \ge k + 1 - e$.
\end{thm}

\subsection{Proof of \cref{t.main} assuming \cref{t.schubert}} \label{ss.reduction}

Assuming \cref{t.schubert} for the while, let us see how it yields \cref{t.main}.

Recalling notation \eqref{e.R_k}, define the quasiprojective variety
$$
\hat{R}_k := R_k \setminus R_{k-1} \, .
$$
We define a map $\pi_k \colon \hat{R}_k \to G_k(\C^n)$
by $A \mapsto \col A$.

\begin{lemma}\label{l.projection}
If $X$ is an algebraically closed column-invariant subset of $\hat{R}_k$
then $Y = \pi_k(X)$ is algebraically closed subset of $G_k(\C^n)$,
and the codimension of $Y$ inside $G_k(\C^n)$ is the same as the codimension of $X$ inside $\hat{R}_k$.
\end{lemma}

\begin{proof}
First, let us see that $\pi_k \colon \hat{R}_k \to G_k(\C^n)$ is a regular map.
We identify $G_k(\C^n)$ with the image of the Pl\"ucker embedding.
In a Zariski neighborhood of each  matrix $A \in \hat{R}_k$, 
the map $\pi_k$ can be defined as $A \mapsto [a_{j_1} \wedge \dots \wedge a_{j_k}]$
for some $j_1 < \dots < j_k$, where $a_j$ is the $j^\text{th}$ column of $A$.
This shows regularity.
	
Next, let us see that $Y = \pi_k (X)$ is closed with respect to the classical (not Zariski)
topology.
Consider the subset $K$ of $X$ formed by the matrices $A \in \hat{R}_k$ 
whose first $k$ columns form an orthonormal set, and whose $m-k$ remaining columns are zero.
Then $K$ is compact (in the classical sense), and thus so is $\pi_k(K)$.
But column-invariance of $X$ implies that $\pi_k(K) = Y$, so $Y$ is closed (in the classical sense).

It follows (see e.g.~\cite[p.39]{Harris}) 
from regularity of $\pi_k$ is regular that the set $Y$ is constructible, i.e., 
it can be written as
$$
Y = \bigcup_{i=1}^{p} Z_i \setminus W_i \, ,
$$
where $Z_i \varsupsetneq W_i$ are algebraically closed subsets of $G_k(\C^n)$.
We can assume that each $Z_i$ is irreducible.
It follows from \cite[Thrm.~2.33]{Mumford} that $\overline{Z_i \setminus W_i} = Z_i$,
where the bar denotes closure in the classical sense.
In particular, $Y = \overline{Y} = \bigcup_{i=1}^{p} Z_i$,
showing that $Y$ is algebraically closed.

We are left to show the equality between codimensions.
Since the codimension of an algebraically closed set equals the minimum of the codimensions of its components, we can assume that $X$ is irreducible.

By column-invariance of $X$,
for each $y\in Y$, the whole fiber $\pi^{-1}(y)$ is contained in $X$.
All those fibers have the same dimension $\mu = km$.
By \cite[Thrm.~11.12]{Harris}, $\dim X = \dim Y + km$.
By \eqref{e.cod_R_k} and \eqref{e.dim_G}, we have $\dim \hat{R}_k - \dim G_k = km$,
so the claim about codimensions follows.
\end{proof}

\begin{proof}[Proof of \cref{t.main}]
Let $X \subset \Mat_{n \times m}(\C)$ be a nonempty algebraically closed, column-invariant set.
Suppose $E$ is a vector subspace of $\C^n$ that does not belong to $\ldbrack X \rdbrack$.
Let $e = \dim E$.
We can assume $e > 0$ (otherwise the result is vacuously true), 
and $e<n$ (because the $e=n$ case was already considered in \cref{s.particular}).

Notice that $X \subset R_{n-1}$.
Let 
$$
X_k := X \cap \hat{R}_k \quad \text{and} \quad 
Y_k := \pi_k(X_k) , \quad \text{for } 0 \le k \le \min(m,n-1).
$$
For every $k$ with $e \le k < n$, the set $Y_k$ is disjoint from 
the set $S_k(E)$ defined by \eqref{e.special schubert}.
In view of \cref{l.projection} and \cref{t.schubert}, we have
$$
\codim_{\hat{R}_k} X_k =  \codim Y_k \ge k + 1 - e \, .
$$
So the codimension of $X_k$ as a subset of $\Mat_{n\times m}(\C)$ is
\begin{align*}
\codim X_k &=   \codim \hat{R}_k + \codim_{\hat{R}_k} X_k \\
           &\ge (m-k)(n-k) + k + 1 - e =: f(k) \, .
\end{align*}
One checks that the function $f(k)$ is decreasing 
on the interval $0 \le k \le \min(m,n-1)$.
Therefore:
\begin{multline*}
\codim X 
=   \min_{0 \le k \le \min(m,n-1)} \codim X_k  
\ge \min_{0 \le k \le \min(m,n-1)} f(k) \\
= f(\min(m,n-1)) 
= m + 1 - e,
\end{multline*}
as claimed.
This proves \cref{t.main} modulo \cref{t.schubert}.
\end{proof}

The proof of \cref{t.schubert} will be given in \cref{s.end},
after we explain the necessary tools in \cref{s.schubert,s.intersection}.

\section{Schubert calculus} \label{s.schubert}

Here we will outline some facts about the intersection of Schubert varieties.
The readable expositions \cite{Blasiak,Vakil} contain more information.

\medskip

A (complete) flag 
in $\C^{n}$ is a sequence of subspaces $F_0 \subset F_1 \subset \cdots \subset F_{n}$
with $\dim F_j = j$. We denote $F_\bullet = \{F_i\}$.

Given $V \in G_k (\C^n)$, 
its \emph{rank table} (with respect to the flag $F_\bullet$)
is the data $\dim (V \cap F_j)$, $j=0,\dots,n$.
The \emph{jumping numbers} are
the indexes $j \in \{1,\dots,n\}$ such that 
$\dim (V \cap F_j) - \dim (V \cap F_{j-1})$ is positive (and thus equal to $1$).
Of course, if one knows the jumping numbers, one know the rank table and vice-versa.
Let us define a third way to encode this information:
Consider a rectangle of height $m$ and width $n-m$, divided in $1 \times 1$ squares.
We form a path of square edges:
Start in the northeast corner of the rectangle.
In the $j^\text{th}$ step ($1 \le j \le n$),
if $j$ is a jumping number then we move one unit in the south direction,
otherwise we move one unit in the west direction.
Since there are exactly $k$ jumping numbers, 
the path ends at the southwest corner of the rectangle.
The \emph{Young diagram} of $V$ with respect to the flag $F_\bullet$ is 
the set of squares in the rectangle that lie northwest of the path.
We denote a Young diagram by $\lambda = (\lambda_1, \lambda_2, \dots, \lambda_k)$,
where $\lambda_i$ is the number of squares in the $i^\text{th}$ row (from north to south).
Its \emph{area} $\lambda_1+\cdots+\lambda_k$ is denoted by $|\lambda|$.

\begin{example}\label{ex.Young}
Here is a possible rank table with $k=5$, $n=12$;
the jumping numbers are underlined:
\begin{center}
\begin{tabular}{rrrrrrrrrrrrrr}
$j = $                 & 0 & 1 & 2 & \underline{3} & 4 & 5 & \underline{6} & 7 & \underline{8} & \underline{9} & 10 & \underline{11} & 12 \\
$\dim (W \cap F_j)=$ & 0 & 0 & 0 &       1 & 1 & 1 &       2 & 2 &        3 &      4 &  4 &        5 &  5 \\
\end{tabular}
\end{center}
The associated path in the rectangle is:
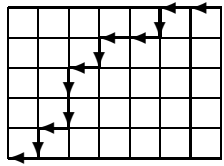
\setlength{\unitlength}{.4cm}
\begin{center}
\raisebox{-3\unitlength}{
\begin{picture}(7,5)
\thinlines
\put(0,0){\grid(7,5)(1,1)}
\thicklines
\put(7,5){\vector(-1,0){1}} 
\put(6,5){\vector(-1,0){1}} 
\put(5,5){\vector(0,-1){1}} 
\put(5,4){\vector(-1,0){1}} 
\put(4,4){\vector(-1,0){1}} 
\put(3,4){\vector(0,-1){1}} 
\put(3,3){\vector(-1,0){1}} 
\put(2,3){\vector(0,-1){1}} 
\put(2,2){\vector(0,-1){1}} 
\put(2,1){\vector(-1,0){1}} 
\put(1,1){\vector(0,-1){1}} 
\put(1,0){\vector(-1,0){1}} 
\end{picture}
}
\end{center}
and so the Young diagram is $$\lambda=\tiny{\yng(5,3,2,2,1)} = (5,3,2,2,1).$$
\end{example}

In general, we have:
\begin{itemize}
\item $\lambda = (\lambda_1, \dots, \lambda_k)$ is a possible Young diagram if and only if
$n-k \ge \lambda_1 \ge \dots \ge \lambda_k \ge 0$.
\item If $j_1 < \dots < j_k$ are the jumping numbers then $\lambda_i = n-k-j_i+i$.
\end{itemize}

The set of $V \in G_k(\C^n)$ that have a given Young diagram 
$\lambda$ 
is called a 
\emph{Schubert cell}, denoted by $\Omega(\lambda)$ or $\Omega(\lambda,F_\bullet)$.
Each Schubert cell is a topological disk of real codimension $2|\lambda|$.
The Schubert cells (for a fixed flag) give a CW decomposition of the space $G_k(\C^n)$.
The closure of $\Omega(\lambda)$ (in either classical or Zariski topologies) is
the set of $V \in G_k(\C^n)$ such that $\dim (V \cap F_{j_i}) \ge i$ for each $i=1,\ldots,n$
(where $j_1 < \dots < j_k$ are the jumping numbers associated to $\lambda$).
These sets are closed irreducible varieties,
called \emph{Schubert varieties}. (See e.g.~\cite[\S9.4]{Fulton}.)

\begin{example}\label{ex.special schubert}
If $E \subset \C^n$ is a subspace with $\dim E = e \le k$ then 
the set $S_k(E)$ defined by \eqref{e.special schubert} is 
a Schubert variety $\bar\Omega(\lambda,F_\bullet)$,
where $F_\bullet$ is any flag with $F_e = E$ and
\setlength{\unitlength}{.25cm}
\begin{equation}\label{e.special young}
\lambda = \big( \underbrace{n-k,\dots,n-k}_{e \text{ times}}, \underbrace{0,\dots,0}_{k-e \text{ times}} \big) = 
\raisebox{-4\unitlength}{
\begin{picture}(12,8)
\thinlines
\put(0,0){\grid(12,8)(1,1)}
\multiput(0,5)(0,1){3}{\multiput(0,0)(1,0){12}{\lightdiagonalpattern}}
\end{picture}
}
\end{equation}
\end{example}

Let $A^*(k,n)$ denote the set of formal linear combinations with integer coefficients of
Young diagrams in the $k \times (n-k)$ rectangle.
This is by definition an abelian group.

\begin{prop}
There is a second $\cupro$ called the \emph{cup product} 
that makes $A^*(k,n)$ a commutative ring, and 
is characterized by the following propertis:

If $\lambda$ and $\mu$ are Young diagrams with respective areas $r$ and $s$
then their cup product is of the form:
$$
\lambda \cupro \mu = \nu_1 + \cdots + \nu_N \, .
$$
where $\nu_1$, \dots, $\nu_N$ are Young diagrams with area $r+s$
(possibly with repetitions, possibly $N=0$).
Moreover, there are flags $F_\bullet$, $G_\bullet$, $H^{(i)}_\bullet$ such that
the manifolds $\bar\Omega(\lambda,F_\bullet)$ and $\bar\Omega(\mu,G_\bullet)$ are transverse
and their intersection is $\bigcup \bar\Omega(\nu_i,H^{(i)}_\bullet)$. 
\end{prop}

\begin{example}
Working in $A^*(2,4)$,
let us compute the products of the Young diagrams 
$\lambda = {\tiny \yng(2)}$ and $\mu={\tiny \yng(1,1)}$.
Fix a flag $F_\bullet$.
Then
$\bar\Omega(\lambda, F_\bullet)$ is the set of $W \in G_2(\C^4)$ that contain $F_1$,
and $\bar\Omega(\mu, F_\bullet)$ is the set of $W \in G_2(\C^4)$ that are contained in $F_3$.
Take another flag $G_\bullet$ which is in general position with respect to $F_\bullet$,
that is $F_i \cap G_{4-i} = \{0\}$.
Then:
\begin{itemize}
\item The set
$\bar\Omega(\lambda, F_\bullet) \cap \bar\Omega(\lambda, G_\bullet)$ 
contains a single element, namely $F_1 \oplus G_1$,
and thus equals $\bar\Omega((2,2),H_\bullet) = \{H_2\}$ for an appropriate flag $H_\bullet$.
This shows that
$\lambda \cupro \lambda = {\tiny \yng(2,2)}$.

\item The space $F_3 \cap G_3$ is $2$-dimensional and thus
is the single element of 
$\bar\Omega(\mu, F_\bullet) \cap \bar\Omega(\mu, G_\bullet)$.
So
$\mu \cupro \mu = {\tiny \yng(2,2)}$.

\item The set 
$\bar\Omega(\lambda, F_\bullet) \cap \bar\Omega(\mu, G_\bullet)$
is empty, thus $\lambda \cupro \mu = 0$.
\end{itemize}
However, if we work in $A^*(4,8)$  
then it can be shown that:
$$
{\tiny \yng(2)}   \cupro {\tiny \yng(2)}   = {\tiny \yng(2,2)} + {\tiny \yng(4)} + {\tiny \yng(3,1)}, \quad
{\tiny \yng(1,1)} \cupro {\tiny \yng(1,1)} = {\tiny \yng(2,2)} + {\tiny \yng(2,1,1)} + {\tiny \yng(1,1,1,1)}, \quad
{\tiny \yng(2)}   \cupro {\tiny \yng(1,1)} = {\tiny \yng(3,1)} + {\tiny \yng(2,1,1)}.
$$
If we drop the terms that do not fit in a $2 \times 2$ rectangle, we reobtain the results for $G_2(\C^4)$. 
\end{example}

The general computation of the product $\lambda \cupro \mu$ is 
not simple and can be done in various ways -- see e.g.\  \cite{Vakil, Fulton}.\footnote{Here is an online calculator: \href{http://young.sp2mi.univ-poitiers.fr/cgi-bin/form-prep/marc/LiE_form.act?action=LRR}{young.sp2mi.univ-poitiers.fr/ cgi-bin/ form-prep/ marc/ LiE{\_}form.act?action=LRR}}
For our purposes, however, it will be sufficient to know when 
the product is zero or not.
The answer is provided by the following simple \lcnamecref{l.overlap}\footnote{In \cite{Vakil} condition~\ref{i.nonoverlap} of the \lcnamecref{l.overlap} is expressed as ``the white checkers are happy''.}:

\begin{lemma}[\cite{Fulton}, p.~148--149]\label{l.overlap}
Let $\lambda$ and $\mu$ be Young diagrams in the $k \times (n-k)$ rectangle.
The following two conditions are equivalent:
\begin{enumerate}

\item\label{i.nonzero} 
$\lambda \cupro \mu \neq 0$.

\item\label{i.nonoverlap} 
If one draws inside the $k \times (n-k)$ rectangle
the Young diagrams of $\lambda$ and $\mu$,
being the later rotated by $180^{\circ}$ and put in the southeast corner,
then the two figures do not overlap
(see \cref{f.nooverlap}). 
Equivalently, $\lambda_i + \mu_{k+1-i} \le n-k$ for every $i=1, \ldots, n$.
\end{enumerate}
\end{lemma}

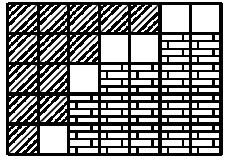
\begin{figure}[hbt]  
\setlength{\unitlength}{.4cm}
\begin{picture}(7,5)

\thicklines
\put(0,0){\grid(7,5)(1,1)}

\thinlines
\put(0,0){\diagonalpattern}
\put(0,1){\diagonalpattern}
\put(1,1){\diagonalpattern}
\put(0,2){\diagonalpattern}
\put(1,2){\diagonalpattern}
\put(0,3){\diagonalpattern}
\put(1,3){\diagonalpattern}
\put(2,3){\diagonalpattern}
\put(0,4){\diagonalpattern}
\put(1,4){\diagonalpattern}
\put(2,4){\diagonalpattern}
\put(3,4){\diagonalpattern}
\put(4,4){\diagonalpattern}

\thinlines
\put(2,0){\brickpattern}
\put(3,0){\brickpattern}
\put(4,0){\brickpattern}
\put(5,0){\brickpattern}
\put(6,0){\brickpattern}
\put(2,1){\brickpattern}
\put(3,1){\brickpattern}
\put(4,1){\brickpattern}
\put(5,1){\brickpattern}
\put(6,1){\brickpattern}
\put(3,2){\brickpattern}
\put(4,2){\brickpattern}
\put(5,2){\brickpattern}
\put(6,2){\brickpattern}
\put(5,3){\brickpattern}
\put(6,3){\brickpattern}

\end{picture}
\caption{{\footnotesize 
Consider $k=5$, $n=12$, $\lambda=(5,3,2,2,1)$, and $\mu = (5,5,4,2,0)$.
The picture shows that the non-overlap condition~(\ref{i.nonoverlap}) from \cref{l.overlap}
is satisfied, and in particular $\lambda \cupro \mu \neq 0$. (This example is reproduced from 
\cite[p.~150]{Fulton}.)}}
\label{f.nooverlap}
\end{figure}


\section{Intersection of subvarieties of the grassmannian}\label{s.intersection}

Next we explain how the Schubert calculus sketched above can be used to obtain information
about intersection of general subvarieties of the Grassmannian,
by means of cohomology and Poincar\'{e} duality.
Our primary source is \cite[Appendix~B]{Fulton};
also, \cite{Hutchings} is a very readable account about the geometric interpretation of the cup product
in cohomology.

\medskip

Any topological space $X$ has singular homology groups $H_i X$ and cohomology groups $H^i X$ (here taken always with integer coefficients). With the cup product $H^i X \times H^j X \to H^{i+j} X$,
the cohomology $H^* X = \bigoplus H^i X$ has a ring structure.

If $X$ is a real compact oriented manifold of dimension $d$ then 
the homology group $H_d X$ is canonically isomorphic to $\Z$, with a generator $[X]$
called the \emph{fundamental class} of $X.$
In addition, there is \emph{Poincar\'{e} duality isomorphism}
$H^i X \to H_{d-i} X$, which is given by 
$\alpha \mapsto \alpha \smallfrown [X]$
(taking the cap product with the fundamental class).
Let us denote by $\omega \mapsto \omega^*$ the inverse isomorphism.

Next suppose $Y$ and $Z$ are compact oriented submanifolds of $X$, of codimensions $i$ and $j$ respectively.
Also suppose that $Y$ and $Z$ have transverse intersection
$Y \cap Z$, which therefore is either empty or a compact submanifold of codimension $i + j$,
which is oriented in a canonical way.
The images of the fundamental classes of $Y$, $Z$, and $Y\cap Z$ under the inclusions into $X$
define homology classes that we denote (with a slight abuse of notation) by
$[Y] \in H_{d-i} X$, $[Z] \in H_{d-j} X$, $[Y\cap Z] \in H_{d-i-j} X$.
Then their Poincar\'e duals 
$[Y]^*\in H^i X$, $[Z]^* \in H^j X$, and $[Y \cap Z]^* \in H^{i+j} X$
are related by: 
$$
[Y]^* \cupro [Z]^* = [Y \cap Z]^* \, .
$$
That is, \emph{cup product is Poincar\'{e} dual to intersection.}

Now consider the case where $X$ is a projective nonsingular (i.e., smooth) complex variety,
and $Y$ and $Z$ are irreducible  subvarieties of $X$.
Obviously, the fundamental class $[X]$ makes sense, because $X$ is a compact manifold
with a canonical orientation induced from the complex structure.
A deeper fact (see \cite[Appendix~B]{Fulton})
is that fundamental classes $[Y]$ and $[Z]$ can also be canonically associated to 
the (possibly singular) subvarieties $Y$ and $Z$,
and the Poincar\'{e} duality between cup product and intersection 
works in this situation.
More precisely,
suppose that $Y$ and $Z$ are transverse in the algebraic sense:
$Y \cap Z$ is a union of subvarieties $W_1$, \dots, $W_\ell$
whose codimensions are the sum of the codimensions of $Y$ and $Z$,
and for each $i=1,\dots,\ell$, the tangent spaces
$T_w Y$ and $T_w Z$ are transverse 
for all $w$ in a Zariski-open subset of $W_i$.
Then each $W_i$ has its canonical fundamental class, 
and the following duality formula holds:
$$
[Y]^* \cupro [Z]^* = [W_1]^* + \cdots + [W_\ell]^* \, .
$$

\medskip

In our application of this machinery, $X$ will be the grassmannian $G_k(\C^n)$.
In this case:
\begin{itemize}
\item The fundamental classes of the Schubert varieties $[\bar\Omega(\lambda, F_\bullet)]$
do not depend on the flag $F_\bullet$.
\item Let $\sigma_\lambda$ denote the Poincar\'e dual of $[\bar\Omega(\lambda, F_\bullet)]$.
Then $H^{2r} G_k(\C^n)$ is a free abelian group and the elements $\sigma_\lambda$
with $|\lambda| = r$ form a set of generators.
(The cohomology groups of odd codimension are zero.)
\item The cup product on cohomology agrees with the ``cup'' product of Young diagrams 
explained in the previous section.
\end{itemize}

\section{End of the proof}\label{s.end}

We are now able to give to prove \cref{t.schubert}.\footnote{Probably the result could also be 
proved using the Chow ring, but we feel more comfortable with cohomology.}

\begin{proof}[Proof of \cref{t.schubert}]
Let $1 \le e \le k < n$.
Let $E \subset \C^n$ be a subspace of dimension $e$,
and consider the set $S_k(E)$ defined by \eqref{e.special schubert}.
Recall from \cref{ex.special schubert} that 
this is the Schubert variety for the Young diagram $\lambda$ given by \eqref{e.special young}.

Now consider a (nonempty) subvariety $Y \subset G_k(\C^n)$
that is disjoint from $S_k(E)$.
We want to give a lower bound for the codimension $c$ of $Y$.
We can of course assume that $Y$ is irreducible.

Let $[Y]^*$ be the dual of fundamental class of $Y$.
This is a nonzero element of $H^{2c} G_k(\C^n)$.
It can be expressed as  
$\sum n_i \sigma_{\mu_i}$,
where $\mu_i$ are Young diagrams with area $|\mu_i|=c$,
and $n_i$ are nonzero integers.
In fact we have $n_i>0$, because of the canonical
orientations induced by complex structure. 

Since the intersection between $S_k(E)$ and $Y$ is empty (and in particular transverse),
Poincar\'{e} duality gives $[S_k(E)]^* \cupro [Y]^* = 0$.
Therefore we have $\sigma_\lambda \cupro \sigma_{\mu_i} = 0$ for each $i$.

By \cref{l.overlap},
if we draw the Young diagram of $\mu_i$
rotated by $180^{\circ}$ and put in the southeast corner
of the $k \times (n-k)$ rectangle,
then it overlaps the Young diagram $\lambda$ pictured in \eqref{e.special young}.
This is only possible if $c \ge k-e+1$;
indeed the Young diagram $\mu$ with least area such that $\lambda \cupro \mu \neq 0$ is
$$
\mu = \big( \underbrace{1,\dots,1}_{k-e+1 \text{ times}}, \underbrace{0,\dots,0}_{e-1 \text{ times}} \big), 
$$
for which the overlapping picture becomes:
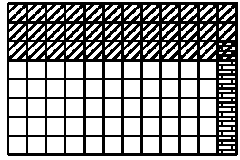
\setlength{\unitlength}{.25cm}
\begin{center}
\begin{picture}(12,8)
\thinlines
\put(0,0){\grid(12,8)(1,1)}
\multiput(0,5)(0,1){3}{\multiput(0,0)(1,0){12}{\lightdiagonalpattern}}
\multiput(11,0)(0,1){6}{\brickpattern}
\thicklines
\put(11,5){\grid(1,1)(1,1)}
\end{picture}
\end{center}
This concludes the proof of  \cref{t.schubert}.
\end{proof}

As explained in \cref{ss.reduction}, \cref{t.main} follows.

%
%
%
%
%
%
%
%


\end{document}